\documentclass[11 pt]{amsart}

\usepackage{amsmath,amsthm,amssymb,amscd,graphicx}

\newtheorem{thm}{Theorem}[section]
 \newtheorem{cor}[thm]{Corollary}
 \newtheorem{lem}[thm]{Lemma}
 
 \newtheorem{prop}[thm]{Proposition}
 
 \theoremstyle{definition}
 
 \theoremstyle{remark}

 \numberwithin{equation}{section}

\def\be#1 {\begin{equation} \label{#1}}
\newcommand{\ee}{\end{equation}}

\newcommand\<{\langle}
\renewcommand\>{\rangle}

\begin{document}

\author{Pierre Germain}
\address{}
\curraddr{}
\email{}

\author{Nader Masmoudi}
\address{}
\curraddr{}
\email{}

\author{Beno\^it Pausader}
\address{}
\curraddr{}
\email{}

\title{Non-neutral global solutions for the electron Euler-Poisson system in $3D$.}

\subjclass[2000]{35Q31 ; 35Q35}

\keywords{}

\begin{abstract}
We prove that small smooth irrotational but charged perturbations of a constant background are global and go back to equilibrium in the $3D$ electron Euler-Poissson equation.
\end{abstract}

\maketitle

\begin{quote}
\footnotesize\tableofcontents
\end{quote}

\section{Introduction}

\subsection{Presentation of the equation}

Consider a three-dimensional plasma composed of fixed ions with uniform density, and a gas of moving electrons. This situation can be described by the Euler-Poisson equation, which 
couples a compressible gas to an electrostatic field. 
Letting $u$ be the velocity of the electron gas and $\rho$ its density, it reads after a simple rescaling
\begin{equation}\label{EP1}
(EP) \quad
\left\{ \begin{array}{l} \partial_t \rho + \nabla \cdot(\rho u) = 0 \\
\rho(\partial_t u + u \cdot \nabla u) = -\rho \nabla \rho + \rho \nabla \Phi \\
         \Delta \Phi = \rho - 1.         \end{array}
\right.
\end{equation}
We took for simplicity the pressure law $p(\rho) = \frac{1}{2} \rho^2$; the analysis is similar for other pressure laws. We make the standing assumption that the fluid is 
irrotational
\begin{equation}\label{IrHyp}
\nabla \times u = 0;
\end{equation}
this condition is of course conserved by the flow of \eqref{EP1}. Our aim is to understand the stability of the obvious stationary state
\begin{equation*}
u = 0 \quad,\quad \rho = 1
\end{equation*}
under perturbations in $\rho$ which do not have mean zero (i.e. are not electrically neutral).

\medskip

It has been proposed that the non-neutral assumption could have important consequences for the asymptotic dynamics of the perturbation. This is made plausible by the following 
remark: consider for each time a ball centered at the center of mass and containing one half of the total electric charge. Far away from the ball, the action of the electric field 
generated by the electrons inside this ball is similar to the action of a single point charge at the center of mass of the same total electric charge. Recall now that particles 
moving in a Coulombian electric field experience asymptotically a logarithmic correction to their free trajectory~\cite{RS}. 
Based on this analogy, one could expect a correction to the 
linear scattering of the perturbation, which had been shown for electrically neutral perturbations (for which the electric field at infinity would be that of a dipole and hence integrable along 
the trajectories).

It turns out that this picture is not accurate, at least for irrotational perturbations. To the best of our understanding, this is due to a combination of the following three facts
\begin{enumerate}
\item The analogy with point particle is not completely accurate.
\item The generated perturbation of the electric field oscillates in time along with the particles. However, the nonlinearities come from inertial terms (convection and pressure) and therefore oscillate  ``out of phase'' with the density and velocity fields (in other words, the nonlinear terms oscillate in a non resonant way with respect to the linear terms). This allows to use a normal form transformation and partially cancel their long-time influence (at least cancel the influence of the ``short-range'' part of the nonlinearity involving $\beta$).
\item The electric field away from the origin has constant amplitude and oscillates (when there is a motion). Its main effect is to periodically repel and attract the particles to the origin. However, since its strength is not uniform, the net effect is to create a repulsive force. To understand the effect of this force, assume that the center of mass remains at the origin. When moving away from the origin, the particles experience a weaker force which takes more time to counteract their inertia and invert their velocity, so they move a long distance away from the origin; in contrast, when the particles move closer to the origin, they encounter a stronger force which sends them back faster and they move a shorter distance towards the origin. However, this dispersion effect depends not on the amplitude of the electric field (which decays like $1/\vert x\vert$), but on the gradient of this field (which decays like $1/\vert x\vert^2$ and is thus integrable). Therefore, it has no long term effect either and we recover linear scattering.
\end{enumerate}

\subsection{Main result}

Our main result is that a constant equilibrium of charged electrons is stable, even under non neutral (but still irrotational) perturbations. We denote
\begin{equation}\label{Charge}
\quad \tilde{Q}=\int_{\mathbb{R}^3}\left[\rho_0(x)-1\right]dx
\end{equation}
(which can thus be taken non zero) for the charge of the perturbation. This extends the work of 
Guo \cite{Guo} who assumed that $\tilde{Q}=0$.

\begin{thm}\label{MainThm1}
There exists $\delta_0$ such that if $(\rho_0,v_0)\in C^\infty_c$ satisfy
\begin{equation*}
\hbox{curl}[v_0]=0,\quad \Vert (\rho_0-1,v_0)\Vert_{H^{10}}+\Vert(\rho_0-1,v_0)\Vert_{W^{5,1}}\le \delta_0,
\end{equation*}
then there exists a unique global solution of \eqref{EP1} which converges to equilibrium in the sense that
\begin{equation*}
\Vert (\rho(t)-1,v(t))\Vert_{L^\infty_x}\to 0,\quad\text{as }t\to+\infty.
\end{equation*}
Furthermore, it scatters in a sense which will be made precise in Corollary~\ref{croco}.
\end{thm}

Again, as in Guo \cite{Guo}, note the contrast with the result in the absence of electric field in \cite{Sid}.
This theorem suggests that the neutral assumption made in \cite{GerMas,IonPau} might be removed. Recall that the neutral assumption is not necessary either to get small-data/global existence for the Euler-Poisson equation for the ions \cite{GuoPau}, which corresponds to the large-time behavior of the system, but the ion case is more transparent since the non-neutral assumption has no implication on the decay of free solutions.


Our analysis relies on works on quasilinear dispersive equations, especially on normal form transform methods, starting from the work of Shatah \cite{Sha} and following recent 
developments of the space-time resonance method in \cite{GerMasSha,GerMasSha2,GNT2,GNT,IonPau}.

The main consequence of the non neutral assumption $\tilde{Q}\ne 0$ is that the solution to the linearized equation is no longer integrable in time. 
However, we remark that its derivative still is. Since the quadratic nonlinearities always involve at least one derivative, the main point is then to systematically track down
the extra decay provided by this derivative term, thus giving a fairly simple proof of Theorem \ref{MainThm1}.

\section{Notations}

We adopt the following notations
\begin{itemize}
\item $A \lesssim B$ if $A \leq C B$ for some implicit constant $C$. The value of $C$ may change from line to line. We note $A\simeq B$ if $A\lesssim B\lesssim A$.
\item If $f$ is a function over $\mathbb{R}^d$ then its Fourier transform, denoted $\widehat{f}$, or $\mathcal{F}(f)$, is given by
\begin{equation*}
\widehat{f}(\xi) = \mathcal{F}f (\xi) = \frac{1}{(2\pi)^{d/2}} \int e^{-ix\xi} f(x) \,dx \;\;\;\;\mbox{thus} \;\;\;\;f(x) = \frac{1}{(2\pi)^{d/2}} \int e^{ix\xi} 
\widehat{f}(\xi) \,d\xi.
\end{equation*}
(in the text, we systematically drop the constants such as $\frac{1}{(2 \pi)^{d/2}}$ since they are not relevant).
\item The Fourier multiplier with symbol $m(\xi)$ is defined by
\begin{equation*}
m(i\nabla)f = \mathcal{F}^{-1} \left[m \mathcal{F} f \right].
\end{equation*}
\item The Littlewood-Paley projector $P_{\le N}$, $P_{\ge N/2}$ and $P_N$ are defined as the Fourier multipliers of symbols
\begin{equation*}
\chi(\xi/N),\quad (1-\chi(\xi/N))\hbox{ and }\chi(\xi/(2N))-\chi(\xi/N)
\end{equation*}
for $\chi\in C^\infty_c(\mathbb{R}^3)$ such that $\chi(x)=1$ when $\vert x\vert\le 1$ and $\chi(x)=0$ when $\vert x\vert\ge 2$. In what follows, 
sums over capital letters are understood to be over dyadic numbers.
\item The bilinear Fourier multiplier with symbol $m$ is given by
\begin{equation}\label{Pseu}
 T_m[f,g](x) \overset{def}{=} \int e^{ix(\xi+\eta)} \widehat{f}(\xi) \widehat{g}(\eta) m(\xi,\eta)\, d\xi d\eta 
= \mathcal{F}^{-1} \int m(\xi-\eta,\eta) \widehat{f}(\xi-\eta) \widehat{g}(\eta)\,d\eta.
\end{equation}
We also define $\tilde{T}_m$ to denote an operator ``of the form'' $T_m$ in the sense that
\begin{equation*}
\tilde{T}_m[f,g]\in\{T_m[f,g],T_m[\overline{f},g],T_m[f,\overline{g}],T_m[\overline{f},\overline{g}]\}.
\end{equation*}
\item The japanese bracket $\langle \cdot \rangle$ stands for $\langle x \rangle = \sqrt{1 + x^2}$.
\item The Riesz transform is defined as the real operator $R_j=\vert\nabla\vert^{-1}\partial_j$.
\item The Besov spaces are defined by their norms as follows
\begin{equation*}
\Vert f\Vert_{B^\sigma_{p,q}}^q=\sum_{N\in 2^\mathbb{Z}}\langle N\rangle^{q\sigma}\Vert P_Nf\Vert_{L^p}^q.
\end{equation*}
\end{itemize}

\section{Preliminary steps}

In order to investigate the stability of $u=0$, $\rho=1$, using \eqref{IrHyp}, we introduce the new unknown function
\begin{equation}\label{NewU}
\alpha = \<\nabla\>\vert\nabla\vert^{-1} (\rho-1) + i \vert\nabla\vert^{-1}\hbox{div}[u].
\end{equation}
The original unknowns can be recovered by the formulas
\begin{equation*}
\rho-1=\vert\nabla\vert\langle\nabla\rangle^{-1}\frak{Re}[\alpha] \quad \mbox{and} \quad u_j = - R_j\frak{Im}[\alpha].
\end{equation*}
The system \eqref{EP1} becomes
\begin{equation}\label{EP2}
\begin{split}
(\partial_t-i\langle\nabla\rangle)\alpha&= - \frac{i}{4}\sum_{j=1}^3R_j\langle\nabla\rangle\big[\frac{|\nabla|}{\langle\nabla\rangle}(\alpha+\overline{\alpha})
\cdot R_j(\alpha-\overline{\alpha})\big]-\frac{i}{8}\sum_{j=1}^3|\nabla|\big[R_j(\alpha-\overline{\alpha})\cdot R_j(\alpha-\overline{\alpha})\big].
\end{split}
\end{equation}
The above right-hand side is a sum of quadratic terms in $\alpha$ and $\bar \alpha$:
\begin{equation*}
RHS = F(\alpha,\alpha) \quad \mbox{where} \quad F(f,g) = Q_1 (f,g) + Q_2 (f,\bar g) + Q_3 (\bar f,\bar g).
\end{equation*}
The bilinear operators $Q_1$, $Q_2$, and $Q_3$ are pseudo-products as in \eqref{Pseu} whose symbols are linear combinations of the following multipliers
\begin{equation}\label{Symb2}
\begin{split}
m_p(\xi_1,\xi_2)&=\vert\xi_1\vert\frac{\langle\xi_1+\xi_2\rangle}{\langle\xi_1\rangle}\frac{\left[ \xi_2\cdot(\xi_1+\xi_2)\right]}{\vert\xi_2\vert\vert\xi_1+\xi_2\vert}=\vert\xi_1\vert\langle\xi_2\rangle\frac{\langle\xi_1+\xi_2\rangle}{\langle\xi_1\rangle\langle\xi_2\rangle}\tilde{m}_p(\xi_1,\xi_2)\\
m_t(\xi_1,\xi_2)&=\vert\xi_1\vert\left[\frac{\xi_1\cdot(\xi_1+\xi_2)}{\vert\xi_1\vert\vert\xi_1+\xi_2\vert}\frac{\xi_1\cdot\xi_2}{\vert\xi_1\vert \vert\xi_2\vert}\right]=\vert\xi_1\vert\tilde{m}_t(\xi_1,\xi_2)
\end{split}
\end{equation}
or their symmetric $m_p(\xi_2,\xi_1)$ and $m_t(\xi_2,\xi_1)$. In \eqref{Symb2}, $m_t$ corresponds to the second term in \eqref{EP2} after using that $\vert\nabla\vert=-\sum_jR_j\partial_j$.

We now isolate the effect of the electric field at infinity. Actually, for simplicity in our analysis, we need to replace $\tilde{Q}$ by $Q$ in \eqref{Q} and introduce $\chi^Q$ as follows\footnote{but one should essentially think of $\chi^Q$ as $Q\vert\nabla\vert^{-1}\chi$ for $\chi$ a nice bump function such that $\int\chi dx=1$.}
\begin{equation*}
\beta(t) \overset{def}{=} \alpha(t) - e^{it\<\nabla\>} \chi^Q \quad \mbox{where} \quad \chi^Q \overset{def}{=} P_{\le 1}\frak{Re}\left[\alpha(0)\right].
\end{equation*}
It solves the system
\begin{equation}\label{EP3}
\begin{split}
\left(\partial_t - i \< \nabla \>\right) \beta&= F(e^{it\<\nabla\>} \chi^Q,e^{it\<\nabla\>} \chi^Q) + \left[F(e^{it\<\nabla\>} \chi^Q,\beta)+F(\beta,e^{it\<\nabla\>} \chi^Q)\right] + F(\beta,\beta)\\
&=I+II+III.
\end{split}
\end{equation}
which is forced by $e^{it\<\nabla\>} \chi^Q$ satisfying (due to~(\ref{crane})):
\begin{equation}\label{BdChi}
\begin{split}
\Vert e^{it\<\nabla\>} \chi^Q\Vert_{B^0_{p,2}}&\lesssim Q(1+t)^{-\frac{3}{2}(1-\frac{2}{p})},\quad 2\le p<3\\
\Vert \nabla e^{it\<\nabla\>} \chi^Q\Vert_{B^0_{p,2}}&\lesssim Q(1+t)^{-\frac{3}{2}(1-\frac{2}{p})},\quad 2\le p<+\infty\\
\Vert e^{it\<\nabla\>} \chi^Q\Vert_{H^N}&\lesssim Q\\
\end{split}
\end{equation}
uniformly in $t\ge 0$, where
\begin{equation}\label{Q}
Q:=\Vert P_{\le 1}\rho(0)\Vert_{L^1}\\ 
\end{equation}
is a substitute for $\tilde{Q}$ in \eqref{Charge}.

Fix $\sigma\ge2$ and $N\ge \sigma+7$.
We define our main norm for the global control:
\begin{equation}\label{Norm}
\begin{split}
\Vert f\Vert_Y&:=\Vert f\Vert_{W^{\sigma+2,10/9}_x}+\Vert f\Vert_{H^N_x}\\
\Vert \beta\Vert_{X_T}&:=\sup_{0\le t\le T}\left[(1+t)^{6/5}\Vert\beta(t)\Vert_{B^{\sigma}_{10,2}}+\Vert \beta(t)\Vert_{H^N_x}\right].
\end{split}
\end{equation}
Using \eqref{crane}, we see that, for all $T$
\begin{equation}\label{LinEst}
\Vert e^{it\langle\nabla\rangle}f\Vert_{X_T}\lesssim \Vert f\Vert_Y
\end{equation}
uniformly in $T$.
Local existence in $X_T$ follows from energy estimates, therefore we see that Theorem \ref{MainThm1} will be a consequence of the following Proposition.

\begin{prop}\label{MainProp}
There exists $\delta>0$ such that if $\beta\in C([0,T]:H^N)$ satisfies \eqref{EP3} on $[0,T]$ and
\begin{equation*}
Q+\sup_{0\le t\le T}\Vert \beta(t)\Vert_{H^5}\le\delta,
\end{equation*}
then 
\begin{equation*}
\Vert\beta\Vert_{X_T}\lesssim \Vert\beta(0)\Vert_Y+(Q+\Vert\beta\Vert_{X_T})^2
\end{equation*}
uniformly in $T$.
\end{prop}

In the proof, we have decided to use Besov spaces instead of more classical spaces in order to have a simple access to the estimates in Lemma \ref{LemBil}. 
More elaborate harmonic analysis techniques probably allow to replace Besov spaces with Lebesgue spaces. In any case, the difference between the two should be thought of as 
inessential to the proof. Also, we have made some effort to quantify the number of derivatives needed and to keep it reasonably low (around $10$), notably through a ``tame'' estimate\footnote{tame in the sense that most of the loss of derivative is on the low frequency term} on bilinear operators in Lemma \ref{LemBil}. 
A slightly more efficient analysis could somewhat reduce this number, but to make it close to the number of derivatives in the physical energy would require significantly stronger results.

\section{Proof of Proposition \ref{MainProp}}

We introduce the linear profile
\begin{equation*}
b(t)=e^{-it\langle\nabla\rangle}\beta(t).
\end{equation*}
This natural unknown is only affected by the nonlinearity
\begin{equation}\label{dtB}
\partial_tb=e^{-it\langle\nabla\rangle}\left[I+II+III\right].
\end{equation}
Using Duhamel formula, we see that
\begin{equation}\label{beta}
\beta(t)=e^{it\langle\nabla\rangle}\left[\beta(0)+\mathcal{N}\right]
\end{equation}
where $\mathcal{N}$ is a finite sum of operators of the form
\begin{equation}\label{DefOfI}
\begin{split}
\mathcal{I}_{\varepsilon,m}[c_1,c_2]&=\int_0^te^{-is\langle\nabla\rangle}T_m[\mathcal{C}^{\varepsilon_1}e^{is\langle\nabla\rangle}c_1(s),\mathcal{C}^{\varepsilon_2}e^{is\langle\nabla\rangle}c_2(s)]ds\\
&=\mathcal{F}^{-1}\int_0^t\int_{\mathbb{R}^3}e^{is\phi_\varepsilon(\xi-\eta,\eta)}m(\xi-\eta,\eta)\widehat{\mathcal{C}^{\varepsilon_1}c_1}(s,\xi-\eta)\widehat{\mathcal{C}^{\varepsilon_2}c_2}(s,\eta)d\eta ds
\end{split}
\end{equation}
for some $m(\xi_1,\xi_2)$ as in \eqref{Symb2}, some $\varepsilon=(\varepsilon_1,\varepsilon_2)\in \{\pm\}^2$, $c_1(t,x),c_2(t,x)\in\{b(t,x),\chi^Q(x)\}$, where $\mathcal{C}^+=Id$ and $\mathcal{C}^-$ denotes the complex conjugation and
\begin{equation}\label{Phase}
\phi_\varepsilon(\xi_1,\xi_2)=-\langle\xi_1+\xi_2\rangle+\varepsilon_1\langle\xi_1\rangle+\varepsilon_2\langle\xi_2\rangle.
\end{equation}
We will obtain bounds which are uniform in $\varepsilon$ and $m$.

\medskip

In all that follows, the worst term to keep in mind is a variation of $e^{it\langle\nabla\rangle}\chi^Q\cdot\vert\nabla\vert P_{\ge 1}e^{it\langle\nabla\rangle}b(t)$ where $b$ has to provide both decay and regularity.

\medskip

We start with a simple estimate.
\begin{lem}\label{ControlDS}
For any choice of $c\in\{b,\chi^Q\}$, we have 
\begin{equation*}
\Vert e^{it\langle\nabla\rangle}\partial_tc(t)\Vert_{H^{\sigma+3}_x}\lesssim (1+t)^{-9/10-\varepsilon}\left[Q+\Vert \beta\Vert_{X_T}\right]^2
\end{equation*}
for $0<\varepsilon<1/100$ and
\begin{equation*}
\Vert e^{it\langle\nabla\rangle}\partial_tc(t)\Vert_{H^{\sigma-1}_x}\lesssim (1+t)^{-6/5}\left[Q+\Vert \beta\Vert_{X_T}\right]^2.
\end{equation*}
\end{lem}

\begin{proof}[Proof of Lemma \ref{ControlDS}]
The case $c=\chi^Q$ is trivial. Hence it suffices to treat the case $c=b$. From \eqref{Symb2}, \eqref{dtB}, it suffices to show that
\begin{equation*}
\begin{split}
\Vert \tilde{T}_{\mu}[\vert\nabla\vert e^{it\langle\nabla\rangle}c_1,\langle\nabla\rangle e^{it\langle\nabla\rangle}c_2]\Vert_{H^{\sigma+3}_x}
&\lesssim (1+t)^{-9/10-\varepsilon}(Q+\Vert\beta\Vert_{X_T})^2,\\
\Vert \tilde{T}_{\mu}[\vert\nabla\vert e^{it\langle\nabla\rangle}c_1,\langle\nabla\rangle e^{it\langle\nabla\rangle}c_2]\Vert_{H^{\sigma-1}_x}
&\lesssim (1+t)^{-6/5}(Q+\Vert\beta\Vert_{X_T})^2
\end{split}
\end{equation*}
for $c_1,c_2\in\{Rb,R\chi^Q\}$, where $R$ denotes a Fourier multiplier coming from a symbol of order $0$, (which we will omit for simplicity of notation) and for 
$\mu\in\{1,\frac{\< \xi_1 + \xi_2 \>}{\< \xi_2 \> \< \xi_1 \>} \}$, for which it is easily seen (for instance by Coifman Meyer theory) that H\"older bounds apply.
We first treat the worst case using \eqref{BdChi} and \eqref{Norm}.
\begin{equation*}
\begin{split}
\Vert \tilde{T}_\mu[\vert\nabla\vert e^{it\langle\nabla\rangle}c_1,\langle\nabla\rangle e^{it\langle\nabla\rangle}\chi^Q]\Vert_{H^{\sigma-1}_x}&\lesssim \Vert \vert\nabla\vert e^{it\langle\nabla\rangle}c_1\Vert_{W^{\sigma-1,10}_x}\Vert\langle\nabla\rangle e^{it\langle\nabla\rangle}\chi^Q\Vert_{W^{\sigma-1,5/2}_x}\\
&\lesssim (1+t)^{-3/2}(Q+\Vert\beta\Vert_{X_T})^2.
\end{split}
\end{equation*}
Independently,
\begin{equation*}
\begin{split}
\Vert \tilde{T}_\mu[\vert\nabla\vert e^{it\langle\nabla\rangle}c_1,\langle\nabla\rangle e^{it\langle\nabla\rangle}\chi^Q]\Vert_{H^{\sigma+6}_x}&\lesssim \Vert \langle\nabla\rangle e^{it\langle\nabla\rangle}\chi^Q\Vert_{L^\infty_x}\Vert c_1\Vert_{H^{\sigma+7}_x}\\
&\lesssim \left[\sum_N\min(N^\frac{1}{2},N^{-1}(1+t)^{-3/2})\right](Q+\Vert \beta\Vert_{X_T})^2\\
&\lesssim (1+t)^{-1/2}(Q+\Vert \beta\Vert_{X_T})^2.
\end{split}
\end{equation*}
Interpolating at order $(4/10+\varepsilon,6/10-\varepsilon)$, we see that this term is acceptable.

It only remains to consider, using \eqref{ProdForm}
\begin{equation*}
\begin{split}
\Vert \tilde{T}_\mu[\vert\nabla\vert \beta,\langle\nabla\rangle\beta]\Vert_{H^{\sigma+4}_x}&\lesssim \Vert\beta\Vert_{W^{2,10}_x}\Vert\beta\Vert_{H^{\sigma+5}_x}\lesssim (1+t)^{-6/5}\Vert \beta\Vert_{X_T}^2.
\end{split}
\end{equation*}
This ends the proof.
\end{proof}

\begin{cor}
\label{croco}
Under the assumptions of the theorem, the solution scatters in $H^{\sigma}$; namely there exists $\alpha_\infty$ in $H^{\sigma}_x$ such that
\begin{equation*}
\left\| \alpha(t) - e^{it \<\nabla\>} \alpha_\infty \right\|_{H^{\sigma}_x} \longrightarrow 0,\quad \mbox{as $t \rightarrow +\infty$} .
\end{equation*}
\end{cor}

\begin{proof}
This follows immediately from the integrability in time of
\begin{equation*}
\Vert \partial_t b(t)\Vert_{H^\sigma_x}\lesssim (1+t)^{-19/14}
\end{equation*}
which in turn follows by interpolation from the two bounds in Lemma \ref{ControlDS}.
\end{proof}

Now, we are ready to prove Proposition \ref{MainProp}.

\begin{proof}[Proof of Proposition \ref{MainProp}]

We start with the dispersive estimate. We claim that under the hypothesis of Proposition \ref{MainProp}, there holds that
\begin{equation}\label{disp}
\sup_{0\le t\le T}(1+t)^\frac{6}{5}\Vert \beta(t)\Vert_{W^{\sigma,10}_x}\lesssim \Vert\beta(0)\Vert_Y+(Q+\Vert\beta\Vert_{X_T})^2.
\end{equation}
Using \eqref{LinEst}, and \eqref{beta}, it suffices to treat the nonlinear term $\mathcal{N}$.
The effect of the nonlinear terms $\mathcal{I}_{\varepsilon,m}$ can be conveniently reformulated through a normal form transformation, which follows from a simple integration 
by parts. Indeed,
\begin{equation}\label{NF}
\begin{split}
\mathcal{I}_{\varepsilon,m}[c_1,c_2](t)
&=-ie^{-it\langle\nabla\rangle}T_{m/\phi_\varepsilon}[\mathcal{C}^\varepsilon_1e^{it\langle\nabla\rangle}c_1(t),\mathcal{C}^\varepsilon_2e^{it\langle\nabla\rangle}c_2(t)]
+iT_{m/\phi_\varepsilon}[c_1(0),c_2(0)]\\
&+i \mathcal{I}_{\varepsilon,m/\phi_\varepsilon}[\partial_sc_1,c_2](t)+ i \mathcal{I}_{\varepsilon,m/\phi_\varepsilon}[c_1,\partial_sc_2](t).
\end{split}
\end{equation}

\medskip

The first term in \eqref{NF} can be treated as follows.
First, from \eqref{LH}, Bernstein and Sobolev inequalities, we observe that
\begin{equation*}
\begin{split}
&\sum_{M\ge 1}\Vert P_{\ge M/8}\tilde{T}_{m/\phi_\varepsilon}[P_{M}e^{it\langle\nabla\rangle}c_1(t),P_{\ge M/8}e^{it\langle\nabla\rangle}c_2(t)]\Vert_{B^{\sigma}_{10,2}}\\
&\lesssim \sum_{ M\ge 1}\Vert e^{it\langle\nabla\rangle}P_{M}c_1(t)\Vert_{W^{1,\infty}_x}\Vert e^{it\langle\nabla\rangle}c_2(t)\Vert_{B^{\sigma+5}_{10,2}}\\
&\lesssim (1+t)^{-6/5}(Q+\Vert\beta\Vert_{X_T}) \sum_{M\ge 1}M^\frac{3}{10}M^{1-\sigma}\Vert c_2\Vert_{H^{\sigma+31/5}_x}\\
&\lesssim (1+t)^{-6/5}(Q+\Vert\beta\Vert_{X_T})^2
\end{split}
\end{equation*}
Independently, using \eqref{LH} with $p=\infty$, we compute that
\begin{equation*}
\begin{split}
&\sum_{M\le 8}\Vert P_{\ge M/8}\tilde{T}_{m/\phi_\varepsilon}[P_Me^{it\langle\nabla\rangle}c_1(t),P_{\ge M/8}e^{it\langle\nabla\rangle}c_2(t)]\Vert_{B^{\sigma}_{10,2}}\\
&\lesssim \sum_{M\le 8}\Vert P_Me^{it\langle\nabla\rangle}c_1(t)\Vert_{L^\infty_x}\Vert e^{it\langle\nabla\rangle}\vert \nabla\vert c_2(t)\Vert_{B^{\sigma}_{10,2}}\\
&\lesssim \left[\sum_{M\le 8}\min(M^{-1/2}(1+t)^{-1},M^{1/2})\right](Q+\Vert\beta\Vert_{X_T})\Vert e^{it\langle\nabla\rangle}\vert \nabla\vert c_2(t)\Vert_{B^{\sigma}_{10,2}}\\
&\lesssim (1+t)^{-1/2}(Q+\Vert\beta\Vert_{X_T})\Vert e^{it\langle\nabla\rangle}\vert\nabla\vert c_2(t)\Vert_{B^{\sigma}_{10,2}}.
\end{split}
\end{equation*}
If $c_2=\chi^Q$, this is acceptable; otherwise, using complex interpolation,
\begin{equation*}
\begin{split}
\Vert e^{it\langle\nabla\rangle}b(t)\Vert_{B^{\sigma+1}_{10,2}}&\lesssim \Vert e^{it\langle\nabla\rangle}b(t)\Vert_{B^{\sigma}_{10,2}}^\frac{7}{12}\Vert e^{it\langle\nabla\rangle}b(t)\Vert_{B^{\sigma+12/5}_{10,2}}^\frac{5}{12}\\
&\lesssim (1+t)^{-7/10}(Q+\Vert\beta\Vert_{X_T})^\frac{7}{12}\Vert b(t)\Vert_{H^{\sigma+18/5}}^\frac{5}{12}\lesssim (1+t)^{-7/10}(Q+\Vert\beta\Vert_{X_T}),
\end{split}
\end{equation*}
and we see that this term is acceptable. Finally, using \eqref{HoldMNO},
\begin{equation*}
\begin{split}
&\sum_{M}\Vert P_{\le M/8}\tilde{T}_{m/\phi_\varepsilon}[P_Me^{it\langle\nabla\rangle}c_1(t),P_{\ge M/8}e^{it\langle\nabla\rangle}c_2(t)]\Vert_{B^{\sigma}_{10,2}}\\
&\lesssim \sum_{N\le M\sim O} N^\frac{3}{5}\Vert P_N\tilde{T}_{m/\phi_\varepsilon}[P_Me^{it\langle\nabla\rangle}c_1(t),P_Oe^{it\langle\nabla\rangle}c_2(t)]\Vert_{W^{\sigma,5}_x}\\
&\lesssim \sum_{N\le M\sim O}N^{\frac{3}{5}}\langle N\rangle^{\sigma+5} \sum_{\{a,b\}=\{1,2\}}\Vert P_Me^{it\langle\nabla\rangle}c_a(t)\Vert_{L^{10}_x}\Vert P_Oe^{it\langle\nabla\rangle}\vert\nabla\vert c_b(t)\Vert_{L^{10}_x}\\
&\lesssim (1+t)^{-6/5}(Q+\Vert\beta\Vert_{X_T})^2
\end{split}
\end{equation*}

\medskip

For the second term, using \eqref{crane}, we see that, for $\kappa=1/1000$,
\begin{equation*}
\begin{split}
\Vert e^{it\langle\nabla\rangle}\tilde{T}_{m/\phi,\varepsilon}[c_1(0),c_2(0)]\Vert_{B^{\sigma}_{10,2}}
&\lesssim (1+t)^{-6/5}\Vert \tilde{T}_{m/\phi,\varepsilon}[c_1(0),c_2(0)]\Vert_{W^{\sigma+2,10/9}_x}\\
&\lesssim (1+t)^{-6/5}\Vert \vert\nabla\vert^{-\kappa}c_1(0)\Vert_{W^{\sigma+5+2\kappa,20/9}}\Vert \vert\nabla\vert^{-\kappa}c_2(0)\Vert_{W^{\sigma+5+2\kappa,20/9}}\\
&\lesssim (1+t)^{-6/5}(Q+\Vert b\Vert_X)^2.
\end{split}
\end{equation*}

\medskip

We treat the third term in \eqref{NF} in a similar way, using \eqref{crane}, Lemma \ref{LemBil} and Lemma \ref{ControlDS}. First, we have for the worst term
\begin{equation*}
\begin{split}
&\sum_{M}\Vert P_{\ge M/8}e^{it\langle\nabla\rangle}\mathcal{I}_{\varepsilon,m/\phi_\varepsilon}[P_{\ge M/8}\partial_sc_1,P_Mc_2]\Vert_{B^{\sigma}_{10,2}}\\
&\lesssim \sum_{M}\int_0^t(1+t-s)^{-6/5}\Vert P_{\ge M/8}\tilde{T}_{m/\phi_\varepsilon}[P_{\ge M/8}e^{is\langle\nabla\rangle}\partial_sc_1,P_Me^{is\langle\nabla\rangle}c_2]\Vert_{B^{\sigma+2}_{10/9,2}}ds\\
&\lesssim \sum_{M}\int_0^t(1+t-s)^{-6/5}\Vert e^{is\langle\nabla\rangle}\partial_sc_1\Vert_{H^{\sigma+3}_x}\Vert P_Me^{is\langle\nabla\rangle}c_2\Vert_{W^{5,\frac{5}{2}}_x}ds\\
&\lesssim \int_0^t(1+t-s)^{-6/5}(1+s)^{-9/10-\varepsilon}\sum_{M}\min(M^\frac{3}{10}(1+M)^{5-N},(1+s)^{-3/10})ds(Q+\Vert\beta\Vert_{X_T})^2\\
&\lesssim (1+t)^{-6/5}(Q+\Vert\beta\Vert_{X_T})^2
\end{split}
\end{equation*}
while the other terms are easier; we now turn to them. Let $\varepsilon$ be as in Lemma \ref{ControlDS} and $2_\varepsilon$ and $(5/2)_\varepsilon$ be such that
\begin{equation*}
\frac{1}{2_\varepsilon}=\frac{1}{2}-\frac{\varepsilon}{3},\quad\frac{1}{(5/2)_\varepsilon}=\frac{2}{5}+\frac{\varepsilon}{3},
\end{equation*}
then
\begin{equation*}
\begin{split}
&\sum_{M}\Vert P_{\ge M/8}e^{it\langle\nabla\rangle}\mathcal{I}_{\varepsilon,m/\phi_\varepsilon}[P_{M}\partial_sc_1,P_{\ge M/8}c_2]\Vert_{B^{\sigma}_{10,2}}\\
&\lesssim \sum_{M}\int_0^t(1+t-s)^{-6/5}\Vert P_{\ge M/8}\tilde{T}_{m/\phi_\varepsilon}[P_{M}e^{is\langle\nabla\rangle}\partial_sc_1,P_{\ge M/8}e^{is\langle\nabla\rangle}c_2]\Vert_{B^{\sigma+2}_{10/9,2}}ds\\
&\lesssim \sum_{M}\int_0^t(1+t-s)^{-6/5}\Vert P_Me^{is\langle\nabla\rangle}\partial_sc_1\Vert_{W^{4,2_\varepsilon}_x}\Vert \vert\nabla\vert e^{is\langle\nabla\rangle}c_2\Vert_{W^{\sigma+3,(5/2)_\varepsilon}_x}ds\\
&\lesssim \int_0^t(1+t-s)^{-6/5}\sum_{M}M^\varepsilon \Vert P_M\partial_sc_1\Vert_{H^5_x}(1+s)^{-3/10+\varepsilon}ds(Q+\Vert\beta\Vert_{X_T})\\
&\lesssim (1+t)^{-6/5}(Q+\Vert\beta\Vert_{X_T})^2
\end{split}
\end{equation*}
and finally,
\begin{equation*}
\begin{split}
&\sum_{M\ge 1}\Vert P_{M}e^{it\langle\nabla\rangle}\mathcal{I}_{\varepsilon,m/\phi_\varepsilon}[P_{\ge 4M}\partial_sc_1,P_{\ge 4M}c_2]\Vert_{B^{\sigma}_{10,2}}\\
&\lesssim \sum_{M\ge 1}\int_0^t(1+t-s)^{-6/5}\Vert P_{M}\tilde{T}_{m/\phi_\varepsilon}[P_{\ge 4M}e^{is\langle\nabla\rangle}\partial_sc_1,P_{\ge 4M}e^{is\langle\nabla\rangle}c_2]\Vert_{B^{\sigma+2}_{10/9,2}}ds\\
&\lesssim \sum_{1\le M\le N\sim O}M^{\sigma+7}N\int_0^t(1+t-s)^{-6/5}\Vert P_N\partial_sc_1\Vert_{L^{2}_x}\Vert P_Oe^{is\langle\nabla\rangle}c_2\Vert_{L^{\frac{5}{2}}_x}ds\\
&\lesssim (1+t)^{-6/5}(Q+\Vert\beta\Vert_{X_T})^2
\end{split}
\end{equation*}
while
\begin{equation*}
\begin{split}
&\sum_{M\le 1}\Vert P_{M}e^{it\langle\nabla\rangle}\mathcal{I}_{\varepsilon,m/\phi_\varepsilon}[P_{\ge 4M}\partial_sc_1,P_{\ge 4M}c_2]\Vert_{B^{\sigma}_{10,2}}\\
&\lesssim \sum_{M\le 1}\int_0^t(1+t-s)^{-6/5}\Vert P_{M}\tilde{T}_{m/\phi_\varepsilon}[P_{\ge 4M}e^{is\langle\nabla\rangle}\partial_sc_1,P_{\ge 4M}e^{is\langle\nabla\rangle}c_2]\Vert_{L^{10/9}_x}ds\\
&\lesssim \sum_{M\le 1}M^{3/10}\int_0^t(1+t-s)^{-6/5}\left[\sum_N\Vert P_N\partial_sc_1\Vert_{L^{2}_x}\Vert P_Nc_2\Vert_{L^{2}_x}\right]ds\\
&\lesssim (1+t)^{-6/5}(Q+\Vert\beta\Vert_{X_T})^2.
\end{split}
\end{equation*}

Now, to finish the proof, we appeal to the following result, which follows from a straightforward adaptation of the computations in \cite{IonPau}.
There exists an energy
\begin{equation*}
E_N:=\sum_{\vert\gamma\vert\le N}\int_{\mathbb{R}^3}\left\{\vert\partial^\gamma(\rho-1)\vert^2+\rho\vert\partial^\gamma v\vert^2+\vert \vert\nabla\vert^{-1}\partial^\alpha(\rho-1)\vert^2\right\}dx
\end{equation*}
such that, if
\begin{equation*}
\sup_{0\le t\le T}\Vert \alpha(t)\Vert_{H^5}\lesssim 1
\end{equation*}
then
\begin{equation*}
E_N(t)\simeq \Vert \alpha(t)\Vert_{H^N}^2
\end{equation*}
and for any $0\le t\le T$,
\begin{equation*}
E_N(t)\lesssim E_N(0)+\int_0^t\Vert\alpha(s)\Vert_{Z^\prime}E_N(s)ds,\quad\Vert f\Vert_{Z^\prime}\lesssim\sup_{M}(M^\frac{3}{4}+M^\frac{4}{3})\Vert P_M\alpha\Vert_{L^\infty}. 
\end{equation*}
Using that
\begin{equation*}
\begin{split}
\Vert P_M\alpha\Vert_{L^\infty_x}&\lesssim \Vert P_Me^{it\langle\nabla\rangle}\chi^Q\Vert_{L^\infty_x}+\Vert P_M\beta\Vert_{L^\infty_x}\\
&\lesssim\min(M^2,1,M^{-1}\langle M\rangle^{\frac{5}{2}} t^{-\frac{3}{2}})Q+\min(M^{-N+3},1,M^\frac{3}{10}\langle M\rangle^{-\sigma}t^{-\frac{6}{5}})\Vert\beta\Vert_X
\end{split}
\end{equation*}
is integrable, we obtain an a priori bound on $E_N$. This ends the proof.

\end{proof}

\section{Appendix: analytical tools}

\subsection{Linear decay}

\label{linear}

The standard dispersive estimates for the linear Klein-Gordon equation follow from a straightforward application of the stationary phase estimates
\begin{equation}
\label{crane}
\left\| e^{it\<\nabla\>} f \right\|_{B^0_{p,2}} \lesssim t^{\frac{3}{p}-\frac{3}{2}} 
\|f\|_{B^{ 5 \left( \frac{1}{2} - \frac{1}{p} \right)}_{ p',2}}\lesssim t^{\frac{3}{p}-\frac{3}{2}} \Vert f\Vert_{W^{5(\frac{1}{2}-\frac{1}{p}),p^\prime}_x}
\;\;\;\;\;\;\mbox{if $2\leq p<\infty$.}
\end{equation}

We also use the simple product formula, valid for all $\gamma\in\mathbb{N}$ and smooth functions $a$, $b$
\begin{equation}\label{ProdForm}
\Vert ab\Vert_{H^\gamma_x}\lesssim_\gamma \Vert a\Vert_{H^\gamma_x}\Vert b\Vert_{W^{1,10}_x}+\Vert a\Vert_{W^{1,10}_x}\Vert b\Vert_{H^\gamma_x}.
\end{equation}

\subsection{Boundedness of pseudo-products}

\label{bilinear}

We need the following Lemma about boundedness of some pseudo-products as defined in \eqref{Pseu}. The main feature here is that the extra loss of derivative coming from the fact that the phase $\phi_\varepsilon$ can be small only gives extra powers of the low frequency $L$.

\begin{lem}\label{LemBil}
Let $m$ be a symbol as in \eqref{Symb2} and $\varepsilon\in\{(\pm,\pm)\}$. Then for any $1\le r\le \infty$ and $1<p,q\le+\infty$ with
\begin{equation*}\label{Holder}
\frac{1}{r}=\frac{1}{p}+\frac{1}{q},
\end{equation*}
we have H\"older's inequality
\begin{equation}\label{HoldMNO}
\Vert P_{M}\tilde{T}_{m/\phi_\varepsilon}[P_Na,P_Ob]\Vert_{L^r_x}\lesssim H(1+L)^5\Vert a\Vert_{L^p_x}\Vert b\Vert_{L^q_x},
\end{equation}
where $H=\max(M,N,O)$ and $L=\min(M,N,O)$.
In particular, we find that for $\kappa>0$,
\begin{equation*}
\Vert T_{m/\phi_\varepsilon}[a,b]\Vert_{W^{\sigma,r}_x}
\lesssim_\kappa \left[\Vert a\Vert_{W^{\sigma+3+\kappa,p}_x}+\Vert \vert\nabla\vert^{-\kappa}a\Vert_{L^p_x}\right]\left[\Vert b\Vert_{W^{\sigma+3+\kappa,q}_x}+\Vert\vert\nabla\vert^{-\kappa}b\Vert_{L^q_x}\right].
\end{equation*}
Furthermore, for $\sigma \geq 0$, $2 \leq p \leq \infty$, $2<q<\infty$, $1<r<\infty$, we get 
\begin{equation}\label{LH}
\Vert P_{\ge M/8}T_{m/\phi_\varepsilon}[P_Ma,P_{\ge M/8}b]\Vert_{B^{\sigma}_{r,2}}\lesssim \Vert a\Vert_{W^{\gamma-\theta,p}_x}\Vert \vert\nabla\vert b\Vert_{B^{\sigma+\theta}_{q,2}}.
\end{equation}
where  $0\le \theta\le\gamma$, $\gamma=5$ if $M\ge 1$ and $\gamma=0$ if $M\le 1$.
\end{lem}

\begin{proof}

Multiplying by a test function, we define
\begin{equation*}\label{Ker}
\begin{split}
I_{M,N,O}=\langle P_Mc,T_{m/\phi_\varepsilon}[P_Na,P_Ob]\rangle&=\iiint_{\mathbb{R}^3}K_{M,N,O}(y_1,y_2,y_3)\overline{a}(y_2)\overline{b}(y_3)c(y_1)dy_1dy_2dy_3\\
K_{M,N,O}(y_1,y_2,y_3)&=\iint_{\mathbb{R}^3}e^{-i\xi\cdot\left[y_1-y_2\right]}e^{-i\eta\cdot\left[y_2-y_3\right]}\varphi(\frac{\xi}{M})\varphi(\frac{\xi-\eta}{N})\varphi(\frac{\eta}{O})\frac{m(\xi-\eta,\eta)}{\phi_\varepsilon(\xi-\eta,\eta)}d\xi d\eta\\
&=\iint_{\mathbb{R}^3}e^{-i\xi\cdot\left[y_1-y_2\right]}e^{-i\eta\cdot\left[y_2-y_3\right]}\frac{\mathfrak{m}_{M,N,O}(\xi-\eta,\eta)}{\phi_\varepsilon(\xi-\eta,\eta)}d\xi d\eta\\
\end{split}
\end{equation*}
Changing variables $(\xi,\eta)\to(\xi,\xi-\eta)$, we may assume that $\min(M,O)/2\le N\le 2\max(M,O)$. 

We compute that
\begin{equation*}
\begin{split}
\left\vert I_{M,N,O}\right\vert&\lesssim \left\vert\sum_{A,B}\iiint_{\vert y_1-y_2\vert\sim A,\,\vert y_2-y_3\vert\sim B}K_{M,N,O}(y_1,y_2,y_3)\overline{a}(y_2)\overline{b}(y_3)c(y_1)dy_1dy_2dy_3\right\vert\\
&\lesssim \sum_{A,B}\iint_{\vert y_1-y_2\vert\sim A}\vert a(y_2)\vert \vert c(y_1)\vert \mathcal{M}b(y_2)\left[\sup_{\vert y_2-y_3\vert\sim B}B^3\vert K_{M,N,O}(y_1,y_2,y_3)\vert\right]dy_1dy_2\\
&\lesssim \left[\sum_{A,B}(AB)^3\sup_{\vert y_1-y_2\vert\sim A,\,\vert y_2-y_3\vert\sim B}\vert K_{M,N,O}(y_1,y_2,y_3)\vert\right]
\int_{\mathbb{R}^3}\vert c(y_1)\vert\cdot \mathcal{M}\left[a\cdot \mathcal{M}b\right](y_1)dy_1,
\end{split}
\end{equation*}
where $\mathcal{M}$ denotes the Maximal function\footnote{If $r=1$, we have to put the maximal function on $c$ instead of on $a\cdot\mathcal{M}b$ at the last line above, but this makes absolutely no change.}.
Using H\"older's inequality and the boundedness of the Maximal function, in order to prove \eqref{HoldMNO}, it suffices to show that
\begin{equation}\label{Sum}
\begin{split}
\sum_{A,B}(AB)^3H^{-1}(1+L)^{-5}c_{A,B}&\lesssim 1,\quad
c_{A,B}=\sup_{\vert y_1-y_2\vert\sim A,\,\vert y_2-y_3\vert\sim B}\vert K_{M,N,O}(y_1,y_2,y_3)\vert.
\end{split}
\end{equation}

We observe that
\begin{equation}\label{BoundM}
\left\vert\partial^\alpha_\xi\partial^\beta_\eta\frac{\mathfrak{m}_{M,N,O}(\xi-\eta,\eta)}{\phi_\varepsilon(\xi-\eta,\eta)}\right\vert\lesssim
\begin{cases}
H\vert\xi\vert^{-\vert\alpha\vert}\vert\eta\vert^{-\vert\beta\vert}&\hbox{ if }L\le 1\\
HL^{-1}\left[\theta+L^{-1}\right]^{-(\vert\alpha\vert+\vert\beta\vert+2)}\vert\xi\vert^{-\vert\alpha\vert}\vert\eta\vert^{-\vert\beta\vert}&\hbox{ if }L\ge 1
\end{cases}
\end{equation}
where $\theta=\vert \angle(\xi,\eta)\vert$.
This follows from the fact that the left-hand side above can be written as a linear combinations of terms like
\begin{equation*}
\frac{\partial_\xi^{\gamma_1}\partial_\eta^{\delta_1}\mathfrak{m}_{M,N,O}}{\phi_\varepsilon}\frac{\partial_\xi^{\gamma_2}\partial_\eta^{\delta_2}\phi_\varepsilon}{\phi_\varepsilon}\dots\frac{\partial_\xi^{\gamma_k}\partial_\eta^{\delta_k}\phi_\varepsilon}{\phi_\varepsilon},\qquad \gamma_i,\delta_i\ge0,\,\,\gamma_1+\dots+\gamma_k=\alpha,\,\,\delta_1+\dots+\delta_k=\beta
\end{equation*}
which is easily seen by induction on $\vert\alpha\vert+\vert\beta\vert$ and from the bounds
\begin{equation*}
\begin{split}
&\vert \phi_\varepsilon(\xi-\eta,\eta)\vert\gtrsim\langle L\rangle\left[\theta^2+\langle L\rangle^{-2}\right],\qquad\vert \partial_\xi^\alpha\partial^\beta_\eta \mathfrak{m}_{M,N,O}\vert\lesssim H\vert\xi\vert^{-\vert\alpha\vert}\vert\eta\vert^{-\vert\beta\vert}\\
&\left\vert\frac{\partial^\alpha_\xi\partial^\beta_\eta\phi_\varepsilon}{\phi_\varepsilon}\right\vert\lesssim
\begin{cases}
\vert\xi\vert^{-\vert\alpha\vert}\vert\eta\vert^{-\vert\beta\vert}&\text{if }\,L\le 1\\
\vert\xi\vert^{-\vert\alpha\vert}\vert\eta\vert^{-\vert\beta\vert}\left[\theta+\vert L\vert^{-1}\right]^{-\vert\alpha\vert-\vert\beta\vert}&\text{if }\,L\ge 1.
\end{cases}
\end{split}
\end{equation*}
These bounds follow from elementary but lengthy computations which we omit\footnote{Remark however that one can easily obtain worst bounds which would allow for a similar H\"older estimate on the pseudo-product but with a worst loss in derivative (but still loosing less than $20$ derivatives). This would be sufficient to obtain the main result assuming more derivatives on the initial data.}.

%
From \eqref{BoundM}, we compute that
\begin{equation*}\label{EstimCoeff}
\begin{split}
A^aB^b\vert c_{AB}\vert&\lesssim \sup_{\vert\alpha\vert=a;\,\vert\beta\vert=b}\left\vert\iint_{\mathbb{R}^3}e^{-i\xi\cdot\left[y_1-y_2\right]}e^{-i\eta\cdot\left[y_2-y_3\right]}\partial^\alpha_\xi\partial^\beta_\eta \frac{\mathfrak{m}_{M,N,O}(\xi-\eta,\eta)}{\phi_\varepsilon (\xi,\eta)} d\xi d\eta\right\vert\\
&\lesssim
\begin{cases}HM^{3-a}O^{3-b}&\hbox{ if }L\le 1,\\
HL^{a+b-1}M^{3-a}O^{3-b}&\hbox{ if }L\ge 1,
\end{cases}
\end{split}
\end{equation*}
from which \eqref{Sum} follows easily, choosing $a,b=2$ or $4$,
\begin{equation*}
\begin{split}
\sum_{A,B}(AB)^3H^{-1}\langle L\rangle^{-5}c_{A,B}&\lesssim \sum_{A,B}(AM/\langle L\rangle)^{3-a}(BO/\langle L\rangle)^{3-b}\lesssim 1.\\
\end{split}
\end{equation*}
This finishes the proof of \eqref{HoldMNO}. Estimate \eqref{LH} then follows directly by summing \eqref{HoldMNO} when $M\sim O\ge N$.
\end{proof}

\end{document}